\documentclass{llncs}
\usepackage{amsmath,amsfonts,amssymb}
\usepackage{graphicx}
\usepackage{epsfig}

\usepackage{color}

\def\e{{\mathbf e}}
\def\x{{\mathbf x}}
\def\y{{\mathbf y}}
\newcommand{\uno}{{\mathbf{1}}}

\newcommand{\R}{\mathbf{R}}
\newcommand{\Z}{\mathbf{Z}}

\newcommand{\thbo}{{\rm TH}}
\newcommand{\stab}{{\rm STAB}}
\newcommand{\LS}{{\rm LS}}
\newcommand{\estab}{{\rm ESTAB}}
\newcommand{\fra}{{\rm ESTAB}}
\newcommand{\qstab}{{\rm QSTAB}}
\newcommand{\astab}{{\rm ASTAB}}

\newcommand{\conv}{\mathrm{conv}}

\DeclareMathOperator{\cone}{cone}
\DeclareMathOperator{\diag}{diag}

\newcommand{\set}[1]{\left\{#1\right\}}

\newtheorem{myclaim}{Claim}

\begin{document}

\title{Lov\'asz-Schrijver PSD-operator  on Claw-Free Graphs
\thanks{This work was supported by 
a MATH-AmSud cooperation (PACK-COVER), PID-CONICET 0277, and PICT-ANPCyT 0586.}}
\author{Silvia Bianchi\inst{1} \and Mariana Escalante\inst{1,2} \and Graciela Nasini\inst{1,2} \and Annegret Wagler\inst{3}\\ \vspace{0.6pt} \email{\{sbianchi,mariana,nasini\}@fceia.unr.edu.ar, wagler@isima.fr}}
\institute{FCEIA, 
Universidad Nacional de Rosario, Rosario, Argentina 
\and CONICET 
\and LIMOS (UMR 6158 CNRS), University Blaise Pascal, 
Clermont-Ferrand, France}

\maketitle

\begin{abstract}
The subject of this work is the study  of  $\LS_+$-perfect graphs defined as those graphs $G$ for which 
the stable set polytope
$\stab(G)$ is achieved in one iteration of  Lov\'asz-Schrijver PSD-operator $\LS_+$, applied to its edge relaxation $\estab(G)$.  In particular, we look for a 
polyhedral relaxation of $\stab(G)$ that coincides with 
$\LS_+(\estab(G))$   
and $\stab(G)$ if and only if $G$ is $\LS_+$-perfect. 
An according conjecture has been recently formulated ($\LS_+$-Perfect Graph Conjecture); here we verify it for the well-studied class of claw-free graphs.

\end{abstract}

%

%

\section{Introduction}

The context of this work is the study of the stable set polytope, some of its linear and semi-definite relaxations, and graph classes for which certain relaxations are tight. 
Our focus lies on those graphs where a single application of the Lov\'asz-Schrijver positive semi-definite operator introduced in \cite{LovaszSchrijver1991} to the edge relaxation yields the stable set polytope. 

The \textit{stable set polytope} $\stab(G)$ of a graph $G=(V,E)$ is defined as the convex hull of the incidence vectors of all stable sets of $G$ (in a stable set all nodes are mutually nonadjacent). 
Two canonical relaxations of $\stab(G)$ are the \textit{edge constraint stable set polytope} 
$$
\estab(G) = \{\mathbf{x} \in [0,1]^V: x_i + x_j \: \le \: 1, ij \in E \}, 
$$
and the \textit{clique constraint stable set polytope} 
$$
\qstab(G) = \{\mathbf{x} \in [0,1]^V: 
\sum_{i \in Q} \, x_i \: \le \: 1,\ Q \subseteq V\ \mbox{ maximal clique of $G$} \}
$$
(in a clique all nodes are mutually adjacent, hence a clique and a stable set share at most one node). We have 
$\stab(G) \subseteq \qstab(G) \subseteq \estab(G)$ 
for any graph, where $\stab(G)$ equals $\estab(G)$ for bipartite graphs, and $\qstab(G)$ for perfect graphs only \cite{Chvatal1975}.  

According to a famous characterization achieved by Chudnovsky et al.~\cite{ChudnovskyEtAl2006}, perfect graphs are precisely the graphs without chordless cycles $C_{2k+1}$ with $k \geq 2$, termed \textit{odd holes}, or their complements, the \textit{odd antiholes} $\overline C_{2k+1}$  as node induced subgraphs 
(where the complement $\overline G$ has the same nodes as $G$, but two nodes are adjacent in $\overline G$ if and only if they are non-adjacent in $G$).  
Then, odd holes and odd antiholes are the only \textit{minimally imperfect graphs}. 

Perfect graphs turned out to be an interesting and important class with a rich structure and a nice algorithmic behavior \cite{GrotschelLovaszSchrijver1988}.  
However, solving the stable set problem for a perfect graph $G$ by maximizing a linear objective function
over $\qstab(G)$ does not work directly \cite{GrotschelLovaszSchrijver1981}, but only via a detour involving a geometric representation of graphs \cite{Lovasz1979} and 
the resulting semi-definite relaxation
$\thbo(G)$ introduced in 
\cite{GrotschelLovaszSchrijver1988}. 

For some $N \in \Z_+$, an orthonormal representation of a graph $G=(V,E)$ is a sequence $(\mathbf{u_i} : i \in V)$ of $|V|$ unit-length vectors $\mathbf{u_i} \in \R^N$, 
such that $\mathbf{u_i}^T\mathbf{u_j} = 0$ for all $ij \not\in E$. For any orthonormal representation of $G$ and any additional unit-length vector $\mathbf{c} \in \R^N$, the 
orthonormal representation constraint is $\sum_{i \in V} (\mathbf{c}^T\mathbf{u_i})^2 x_i \leq 1$. $\thbo(G)$ denotes the convex set of all vectors $\mathbf{x} \in \R_+^{|V|}$ satisfying all orthonormal representation constraints for $G$. 
For any graph $G$, 
\[
\stab(G) \subseteq \thbo(G) \subseteq \qstab(G)
\]
holds and approximating a linear objective function over $\thbo(G)$ can be done with arbitrary precision in polynomial time \cite{GrotschelLovaszSchrijver1988}.  Moreover, if $\thbo(G)$ is a rational polytope, an optimal solution can be obtained in polynomial time. This fact gives a great relevance to the
beautiful characterization of perfect graphs  obtained by the same authors: 
\begin{equation} \label{equivalencias}
G \textrm{ is perfect} \Leftrightarrow \thbo(G)=\stab(G) \Leftrightarrow \thbo(G)=\qstab(G).
\end{equation}
%
For all imperfect graphs, 
$\stab(G)$ does not coincide with any of the above relaxations. It is, thus, natural to study further relaxations and to combinatorially characterize those graphs where $\stab(G)$ equals one of them.

\paragraph{Linear relaxations and related graphs.} 
A natural generalization of the clique constraints 
are rank constraints associated with arbitrary induced subgraphs $G' \subseteq G$. 
By the choice of the right hand side $\alpha(G')$, denoting the size of a largest stable set in $G'$, rank constraints 
$$
\mathbf{x}(G') = \sum_{i \in G'} \, x_i \, \le \, \alpha(G') 
$$ 
are  valid for $\stab(G)$.

A graph $G$ is called \textit{rank-perfect} by \cite{Wagler2000} if and only if $\stab(G)$ is described by rank constraints only. 

By definition, rank-perfect graphs include all perfect graphs. 
By restricting the facet set to rank constraints associated with certain subgraphs, several well-known graph classes are defined, e.g., 
\textit{near-perfect graphs} \cite{Shepherd1994} where only rank constraints associated with cliques and the whole graph are  allowed, or \textit{t-perfect} \cite{Chvatal1975} and \textit{h-perfect graphs} \cite{GrotschelLovaszSchrijver1988} where rank constraints associated with edges, triangles and odd holes resp. cliques of arbitrary size and odd holes suffice. 

As common generalization of perfect, t-perfect, and h-perfect graphs, the class of \textit{a-perfect graphs} was introduced in \cite{Wagler2004_4OR} as graphs $G$ where $\stab(G)$ is given by 
rank constraints associated with antiwebs. 
An \textit{antiweb} $A^k_n$ is a graph with $n$ nodes $0, \ldots, n-1$ and edges $ij$ if and only if $k \leq |i-j| \leq n-k$ and $i \neq j$. 
Antiwebs include all complete graphs $K_n = A^1_n$, odd holes $C_{2k+1} = A^k_{2k+1}$, and their complements $\overline C_{2k+1} = A^2_{2k+1}$. 
Antiwebs are $a$-perfect by \cite{Wagler2004_4OR}, 
further examples of $a$-perfect graphs were found in \cite{Wagler2005_4OR}.

A more general type of inequalities is obtained from complete joins of antiwebs, called \textit{joined antiweb constraints}
\[
\sum_{i \leq k} \frac{1}{\alpha(A_i)} x(A_i) + x(Q) \leq 1, 
\]
associated with the join of some antiwebs $A_1, \ldots, A_k$ and a clique $Q$ 
(note that the inequality is scaled to have right hand side 1). This includes, e.g., all odd (anti)wheels (the join of a single node with an odd (anti)hole). 
We denote the linear relaxation of $\stab(G)$ obtained by all joined antiweb constraints by $\astab^*(G)$. By construction, we see that 
$$\stab(G) \subseteq \astab^*(G) \subseteq \qstab(G) \subseteq \estab(G).$$
In \cite{CPW_2009}, a graph $G$ is called \textit{joined a-perfect} if and only if $\stab(G)$ coincides with $\astab^*(G)$. 
Besides all a-perfect graphs, further examples of joined a-perfect graphs are 
\textit{near-bipartite graphs} (where the non-neighbors of every node induce a bipartite graph) due to \cite{Shepherd1995}. 

\paragraph{A semi-definite relaxation and $\LS_+$-perfect graphs.} 
In the early nineties, Lov\'{a}sz and Schrijver introduced the PSD-operator $\LS_+$ 
(called $N_+$ in \cite{LovaszSchrijver1991}) which, applied to 
$\estab(G)$, generates a positive semi-definite relaxation of $\stab(G)$ stronger than $\thbo(G)$ (see Section \ref{Sec:N+perfect} for details). 
In order to simplify the notation we write $\LS_+(G)=\LS_+(\fra(G))$. 

As in the case of perfect graphs, the stable set problem can be solved in polynomial time for the class of graphs for which $\LS_+(G)=\stab(G)$ by \cite{LovaszSchrijver1991}. 
These graphs are called \emph{$\LS_+$-perfect}, and all other graphs \emph{$\LS_+$-imperfect} (note that they are also called $N_+$-(im)perfect, see e.g. \cite{BENT2011}). 

In addition, 
every subgraph of an $\LS_+$-perfect graph is also $\LS_+$-perfect. 
This motivates the definition of \emph{minimally} $\LS_+$-\emph{imperfect graphs} as the $\LS_+$-imperfect graphs whose proper induced subgraphs are all $\LS_+$-perfect. 
The two smallest of such graphs (regarding its number of nodes) were found by \cite{EMN2006} and \cite{LiptakTuncel2003} and  are depicted in Figure \ref{grafos6}. 

\begin{figure}
\begin{center}
\includegraphics[scale=1.0]{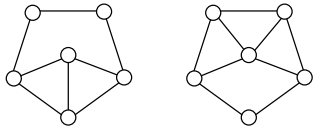}
\caption{The graphs $G_{LT}$ (on the left) and $G_{EMN}$ (on the right).}
\label{grafos6}
\end{center}
\vspace*{-5mm}
\end{figure}

In \cite{BENT2011}, 
the authors look for a characterization of $\LS_+$-perfect graphs similar to the characterization (\ref{equivalencias}) for perfect graphs: they intend to find an appropriate polyhedral relaxation 
$P(G)$ of $\stab(G)$ such that 
$G$ is $\LS_+$-perfect if and only if $\stab(G)=P(G)$. 
A conjecture has been recently proposed in \cite{BENT2013}, which can be equivalently reformulated as follows \cite{ENW2014}:

\begin{conjecture}[$\LS_+$-Perfect Graph Conjecture] \label{conjecture}
A graph $G$ is $\LS_+$-perfect if and only if  $\stab(G)=\astab^*(G)$.
\end{conjecture}

In \cite{LovaszSchrijver1991} it is shown that, for every graph $G$ 
\begin{equation}\label{LS1}
\LS_+(G) \subseteq \astab^*(G).
\end{equation}

Thus, the conjecture states that $\LS_+$-perfect graphs coincide with joined a-perfect graphs 
and $\astab^*(G)$ is the 
 polyhedral relaxation of $\stab(G)$ playing the role of $P(G)$ in \eqref{equivalencias}.

Conjecture~\ref{conjecture} has been already verified for near-perfect graphs by \cite{BENT2011}, for \emph{fs-perfect graphs} (where the only facet-defining subgraphs are cliques and the graph itself) by \cite{BENT2013}, for \emph{webs} (the complements $W^k_n = \overline A^k_n$ of antiwebs) by \cite{EN2014} and for \emph{line graphs} (obtained by turning adjacent edges of a root graph into adjacent nodes of the line graph) by \cite{ENW2014}, see Section \ref{Sec:N+perfect} for details.

\paragraph{The $\LS_+$-Perfect Graph Conjecture for Claw-free graphs.}

The aim of this contribution is to verify Conjecture \ref{conjecture} for a well-studied graph class containing all webs, all line graphs and the complements of near-bipartite graphs: the class of \emph{claw-free graphs} (i.e., the graphs not containing  as node induced subgraph  the complete join of a single node and a stable set of size three). 

Claw-free graphs attracted much attention due to their seemingly asymmetric behavior w.r.t. the stable set problem. 
On the one hand, the first combinatorial algorithms to solve the problem in polynomial time for claw-free graphs \cite{Minty1980,Sbihi1980} date back to 1980. 
Therefore, the polynomial equivalence of optimization and separation due to \cite{GrotschelLovaszSchrijver1988} implies that it is possible to optimize over the stable set polytope of a claw-free graph in polynomial time. 
On the other hand, the problem of characterizing the stable set polytope of claw-free graphs in terms of an explicit description by means of a facet-defining system, originally posed in~\cite{GrotschelLovaszSchrijver1988}, was open for several decades. 
This motivated the study of claw-free graphs and its different subclasses, that finally answered this long-standing problem only recently (see Section \ref{Sec:ClawFree} for details).

The paper is organized as follows: 
In Section 2, we present the State-of-the-Art on $\LS_+$-perfect graphs (including families of $\LS_+$-imperfect graphs needed for the subsequent proofs) and on claw-free graphs, their relevant subclasses and the results concerning the facet-description of their stable set polytopes from the literature. 
In Section 3, we verify, relying on the previously presented results, Conjecture \ref{conjecture} for the studied subclasses of claw-free graphs. As a conclusion, we obtain as our main result:

\begin{theorem}\label{thm_main}
The $\LS_+$-Perfect Graph Conjecture is true for all claw-free graphs. 
\end{theorem}

We close with some further remarks and an outlook to future lines of research. 

\section{State-of-the-Art}
\label{Sec:N+perfect}

\subsection{About $\LS_+$-perfect graphs}

In order to introduce the $\LS_+$-operator we 
denote by $\e_0, \e_1, \dots, \e_n$ the vectors of the canonical basis of $\R^{n+1}$ (where  the first coordinate is indexed zero), $\uno$ the vector with all components equal to $1$ and $S_+^{n}$ the convex cone of symmetric and positive semi-definite $(n \times n)$-matrices with real entries.   
Let $K\subset [0,1]^n$ be  a  convex set and
\[
 \cone(K)= \left\{\left(\begin{array}{c} x_0\\ \x \end{array}
 \right) \in \R^{n+1}: \x=x_0 \y; \;\;  \y \in K \right\}.
 \]
Then, the convex set $M_+(K)$ is defined as:
 \begin{eqnarray*}
 M_{+}(K) = \left\{Y \in S_+^{n+1}: \right.
 & & Y\e_0 = \diag (Y),\\
 & & Y\e_i \in \cone(K), \\ 
 & & \left. Y (\e_0 - \e_i) \in \cone(K), \; 
 i=1,\dots,n \right\},
 \end{eqnarray*}
where $\diag (Y)$ denotes the vector whose $i$-th entry is $Y_{ii}$, for every $i=0,\dots,n$.  
Projecting this lifting back to the space $\R^n$ results in 
 \[
 \LS_{+}(K) = \set{ \x \in [0,1]^n : \left(\begin{array}{c} 1\\ \x \end{array}
 \right)= Y \e_0, \mbox{ for some } Y \in M_{+}(K)}.
 \]
In \cite{LovaszSchrijver1991}, Lov\'{a}sz and Schrijver proved that $\LS_+(K)$ is a relaxation of  the convex hull of integer solutions in $K$ and that 
$$\LS_+^n(K)=\conv(K\cap \{0,1\}^n),$$ 
where $\LS_+^0(K)=K$ and $\LS_+^k(K)=\LS_+(\LS_+^{k-1}(K))$ for every $k\geq 1$. 

In this work we focus on the behavior of a single application of the $\LS_+$-operator to the edge relaxation $\fra(G)$ of the stable set polytope of a graph.

Recall that we write $\LS_+(G)=\LS_+(\fra(G))$ to simplify the notation and that graphs for which $\LS_+(G)=\stab(G)$ holds are $\LS_+$-perfect.

Exhibiting one $\LS_+$-imperfect subgraph $G'$ in a graph $G$ certifies the $\LS_+$-imperfection of $G$. 
Hereby, characterizing $\LS_+$-imperfect graphs within a certain graph class turns out to be a way to attack the conjecture for this class.

Recall that $G_{LT}$ and $G_{EMN}$ are the smallest $\LS_+$-imperfect graphs. 
In \cite{BENT2011} the authors showed that they are the two smallest members of an infinite family of $\LS_+$-imperfect graphs having stability number two that will play a central role in some subsequent proofs:

\begin{theorem} [\cite{BENT2011}] \label{mnpmascero}
Let $G$ be a graph with $\alpha(G)=2$ such that $G-v$ is an odd antihole for some node $v$. $G$ is $\LS_+$-perfect if and only if $v$ is completely joined to $G-v$. 
\end{theorem}

Further $\LS_+$-imperfect graphs can be obtained by applying operations preserving $\LS_+$-imperfection.

In \cite{LiptakTuncel2003}, the \emph{stretching} of a node $v$ is introduced as follows: 
Partition its neighborhood $N(v)$ into two nonempty, disjoint sets $A_1$ and $A_2$ (so $A_1 \cup A_2 = N(v)$, and $A_1 \cap A_2 = \emptyset$). 
A stretching of $v$ is obtained by replacing $v$ by two adjacent nodes $v_1$ and $v_2$, joining $v_i$ with every node in $A_i$ for $i \in \{1, 2\}$, and 
subdividing the edge $v_1 v_2$ by one node $w$. In \cite{LiptakTuncel2003} it is shown:

\begin{theorem} [\cite{LiptakTuncel2003}] \label{stretching}
The stretching of a node preserves $\LS_+$-imperfection.  
\end{theorem}

Hence, all stretchings of 
$G_{LT}$ and $G_{EMN}$ are $\LS_+$-imperfect, see Figure \ref{fig_stretchings} for some examples.

\begin{figure}
\begin{center}
\includegraphics[scale=1.0]{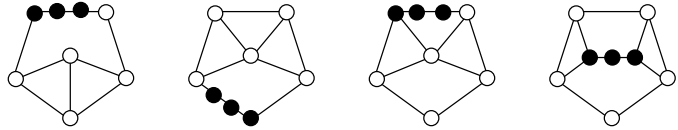}
\caption{Some node stretchings ($v_1,w,v_2$ in black) of $G_{LT}$ and $G_{EMN}$.}
\label{fig_stretchings}
\end{center}
\end{figure}

Using stretchings of $G_{LT}$ and $G_{EMN}$ and exhibiting one more minimally $\LS_+$-imperfect graph, namely the web $W_{10}^2$, $\LS_+$-perfect webs are characterized in \cite{EN2014} as follows:

\begin{theorem} [\cite{EN2014}] \label{webs}
A web is $\LS_+$-perfect if and only if it is 
perfect or minimally imperfect.
\end{theorem}

The proof shows that all imperfect not minimally imperfect webs with stability number 2  contain $G_{EMN}$ and all webs $W_{n}^2$ different from $W_{7}^2, W_{10}^2$, some stretching of $G_{LT}$. Furthermore, all other webs contain some $\LS_+$-imperfect $W_{n'}^2$ and are, thus, also $\LS_+$-imperfect. 

Another way to attack the conjecture is from the polyhedral point of view. 

A graph $G$ is said to be \emph{facet-defining} if $\stab(G)$ has a full-support facet. 
Observe that verifying Conjecture \ref{conjecture} is equivalent to prove that the only facet defining $\LS_+$-perfect graphs are the complete joins of antiwebs. That is why we need to rely on structural results and complete facet-descriptions of  stable set polytope of the graphs. 

Using this approach, in \cite{ENW2014}, the authors characterized $\LS_+$-perfect line graphs by showing: 

\begin{theorem} [\cite{ENW2014}] \label{hipomatch}
A facet-defining line graph $G$ is $\LS_+$-perfect if and only if $G$ is a clique or an odd hole. 
\end{theorem}

The proof relies on a result  due to Edmonds \& Pulleyblank~\cite{EdmondsPulleyblank1974} who showed that a line graph $L(H)$ is facet-defining if and only if $H$ is a 2-connected hypomatchable graph (that is, for all nodes $v$ of $H$, 
$H-v$ admits a perfect matching). 
Such graphs $H$ have an ear decomposition $H_0, H_1, \dots, H_k = H$ where $H_0$ is an odd hole and $H_i$ is obtained from $H_{i-1}$ by adding an odd path (ear) between distinct nodes of $H_{i-1}$. 
In \cite{ENW2014}, it is shown that the line graph $L(H_1)$
 is a node stretching of $G_{LT}$ or $G_{EMN}$ and, thus, $\LS_+$-imperfect by \cite{LiptakTuncel2003}. 

Moreover, it is proved that the only minimally $\LS_+$-imperfect line graphs are stretchings of $G_{LT}$ and $G_{EMN}$.

\subsection{About claw-free graphs}
\label{Sec:ClawFree}

In several respects, claw-free graphs are generalizations of line graphs. 
An intermediate class between line graphs and claw-free graphs form \textit{quasi-line graphs}, where the neighborhood of any node can be partitioned into two cliques (i.e., quasi-line graphs are the complements of near-bipartite graphs). 

Quasi-line graphs can be divided into two subclasses: fuzzy circular interval graphs and semi-line graphs.

Let ${\cal C}$ be a circle, ${\cal I}$ a collection of intervals 
in ${\cal C}$ without proper containments and common endpoints, 
and $V$ a multiset of points in ${\cal C}$. 
 A \textit{fuzzy circular interval graph} $G(V,{\cal I})$ 
has node set $V$ and two nodes are adjacent 
if both belong to one interval $I \in {\cal I}$, where 
edges between different endpoints of the same interval may be omitted. 

{\em Semi-line graphs} are either line graphs or quasi-line graphs without a representation as a fuzzy circular interval graph.

It turned out that so-called \emph{clique family inequalities} suffice to describe the stable set polytope of quasi-line graphs. 
Given a graph $G$, a family $\mathcal F$ of cliques and an integer $p<n=|\mathcal F|$, the clique family inequality ($\mathcal F$, $p$) is the following valid inequality for $\stab(G)$
\begin{equation}
(p-r)\sum_{i\in W} x_i+ (p-r-1) \sum_{i\in W_o} x_i \leq (p-r) \left\lfloor \frac{n}{p}\right\rfloor
\end{equation}
where $r=n \, mod\, p$ and $W$ (resp. $W_o$) is the set of nodes contained in at least $p$ (resp. exactly $p-1$) cliques of $\mathcal F$. 

This 
generalizes the results of Edmods~\cite{Edmonds1965} and Edmonds \& Pulleyblank~\cite{EdmondsPulleyblank1974} that $\stab(L(H))$ is described by clique constraints 
and rank constraints  
\begin{equation}
\label{Eq_hypomatch-constraints}
x(L(H')) \le \frac{1}{2}(|V(H')|-1)
\end{equation}
associated with the line graphs of 2-connected hypomatchable induced subgraphs $H' \subseteq H$. 
Note that the rank constraints of type (\ref{Eq_hypomatch-constraints}) are special clique family inequalities. 

Chudnovsky and Seymour \cite{ChudnovskySeymour2004} extended this result to 
semi-line graphs, for which $\stab(G)$ is given by clique constraints and rank constraints of type (\ref{Eq_hypomatch-constraints}). Then, 
semi-line graphs are rank-perfect with line graphs as only facet-defining subgraphs. 

Moreover, in \cite{GalluccioSassano97} Galluccio and Sassano prove that if a rank constraint is facet-defining for a claw-free graph $G$ then, either, $G$ is a clique or
$G$ contains the line graph of a minimal 2-connected hypomatchable graph $H$ or $G$ contains $W_{\alpha k +1}^{k-1}$ with $k\geq 3$ and $\alpha =\alpha(G)$.

Eisenbrand et al.~\cite{EisenbrandEtAl2005} proved that clique family inequalities 
suffice to describe the stable set polytope of fuzzy circular interval graphs. 
Stauffer \cite{Stauffer} verified a conjecture of \cite{PecherWagler2006} that every facet-defining clique family inequality of a fuzzy circular interval graph $G$ is \emph{associated with a web in} $G$.

All these results together complete the picture for quasi-line graphs. 

However, there are claw-free graphs which are not quasi-line. In particular, every graph with stability number 2 is claw-free and the 5-wheel is the smallest claw-free not quasi-line graph.

Due to Cook (see~\cite{Shepherd1995}), all facets for graphs $G$ with $\alpha(G)=2$ 
are $1,2$-valued {\em clique-neighborhood constraints}.
This is not the case for graphs $G$ with $\alpha(G) = 3$. 
In fact, all the known difficult facets of claw-free graphs occur in this class.   
Some non-rank facets with up to five different non-zero coefficients are presented in~\cite{GilesTrotter1981,LieblingEtAl2004}. 
All of these facets turned out to be so-called {\em co-spanning 1-forest constraints} due to \cite{PecherWagler2010}, where it is also shown that it is possible to build a claw-free graph with stability number three inducing a co-spanning 1-forest facet with $b$ different left hand side coefficients, for every positive integer $b$. 

The problem of characterizing $\stab(G)$ when $G$ is a connected claw-free but not quasi-line graph with $\alpha(G) \geq 4$ was studied by Galluccio et al.: 
In a series of results \cite{GGV2008,GGV2014a,GGV2014b}, it is shown that if such a graph $G$ does not contain a clique cutset, then 1,2-valued constraints suffice to describe STAB($G$). Here, besides 5-wheels, different rank and non-rank facet-defining inequalities of the geared graph $G$ shown in Fig.~\ref{Fig:5wheelStrips} play a central role. 

In addition, graphs of this type can be decomposed into strips. 
A \emph{strip} $(G,a,b)$ is a (not necessarily connected) graph with two designated simplicial nodes $a$ and $b$ (a node is \emph{simplicial} if its neighborhood is a clique). 
A claw-free strip containing a 5-wheel as induced subgraph is a {\em 5-wheel strip}. 
Given two node-disjoint strips $(G_1,a_1,b_1)$ and $(G_2,a_2,b_2)$, their \emph{composition} is the union of $G_1\setminus\{a_1,b_1\}$ and $G_2\setminus\{a_2,b_2\}$ together with all edges between $N_{G_1}(a_1)$ and $N_{G_2}(a_2)$, and between $N_{G_1}(b_1)$ and $N_{G_2}(b_2)$ \cite{ChudnovskySeymour2004}. 
 
As shown in~\cite{OPS2008}, this composition operation can be generalized to more than two strips: 
Every claw-free but not quasi-line graph $G$ with $\alpha(G) \geq 4$ admits a decomposition into strips, where at most one strip is quasi-line and all the remaining ones are  5-wheel strips having stability number at most 3. 
There are only three ``basic'' types of 5-wheel strips (see Fig.~\ref{Fig:5wheelStrips}) which can be extended by adding nodes belonging to the neighborhood of the 5-wheels (see~\cite{OPS2008} for details). 

Note that a claw-free but not quasi-line graph $G$ with $\alpha(G) \geq 4$ containing a clique cutset may have a facet-defining subgraph $G'$ with $\alpha(G') = 3$ (inside a 5-wheel strip of type 3), 
see~\cite{PietropaoliWagler2008} for examples. 

Taking all these results together into account gives the complete list of facets needed to describe the stable set polytope of claw-free graphs.

\begin{figure}
\begin{center}
\includegraphics[scale=0.8]{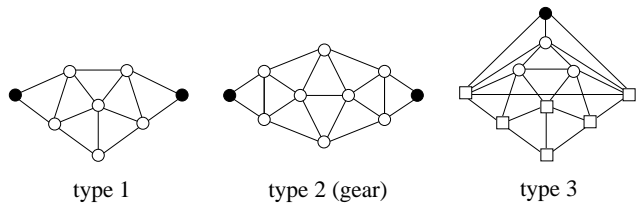}
\caption{The three types of basic 5-wheel strips.}
\label{Fig:5wheelStrips}
\end{center}
\end{figure}

\section{$\LS_+$-Perfect Graph Conjecture for claw-free graphs}
\label{results}

In this section, we verify the $\LS_+$-Perfect Graph Conjecture for all relevant subclasses of claw-free graphs.

\subsection{Graphs with $\alpha(G) = 2$}

The graphs with $\alpha(G) = 2$ play a crucial role 
in this context.
Relying on the behavior of the stable set polytope under taking complete joins \cite{Chvatal1975} and the result on $\LS_+$-(im)perfect graphs $G$ with $\alpha(G) = 2$ (Theorem \ref{mnpmascero}), we can prove:
 
\begin{theorem}
\label{thm_N+imperfect_alpha2}
All facet-defining $\LS_+$-perfect graphs $G$ with $\alpha(G) = 2$ are odd antiholes or complete joins of odd antihole(s) and a (possibly empty) clique.
\end{theorem}

\begin{proof}
Let $G$ be a facet-defining $\LS_+$-perfect graph with stability number 2. 
We first observe that $G$ is imperfect (because it has a full-support facet, but it is different from a clique). 
Thus, $G$ contains an odd antihole $\overline C$ by \cite{ChudnovskyEtAl2006}.

If $G = \overline C$, we are done. 
If $G \neq \overline C$, every node $u$ outside $\overline C$ is completely joined to $\overline C$ due to Theorem \ref{mnpmascero} (otherwise, $u$ and $\overline C$ induce an $\LS_+$-imperfect subgraph of $G$, a contradiction to $G$ $\LS_+$-perfect). 

Therefore, $G$ is the complete join of $\overline C$ and $G-\overline C$. 
Note that $G-\overline C$ is again $\LS_+$-perfect, facet-defining by Chv\'atal \cite{Chvatal1975}, and  $\alpha(G-\overline C) \leq 2$. 
If $\alpha(G-\overline C) = 2$, we apply the same argument as for $G$; if $\alpha(G-\overline C) = 1$, it is a clique. $\Box$
\end{proof}

This shows that all facet-defining $\LS_+$-perfect graphs $G$ with $\alpha(G) = 2$ are joined $a$-perfect, and we conclude:

\begin{corollary}\label{Cor_alpha2_2}
The $\LS_+$-Perfect Graph Conjecture is true for graphs with stability number 2.
\end{corollary}

\subsection{Quasi-line graphs}

Recall that quasi-line graphs 
divide into the two subclasses of semi-line graphs and fuzzy circular interval graphs. 

Chudnovsky and Seymour \cite{ChudnovskySeymour2004} proved that the stable set polytope of a semi-line graph is given by rank constraints associated with cliques and the line graphs of 2-connected hypomatchable graphs. Together with the result from \cite{ENW2014} (presented in Theorem \ref{hipomatch}), we directly conclude that 
the $\LS_+$-Perfect Graph Conjecture holds for semi-line graphs. 

Based on the results of Eisenbrand et al.~\cite{EisenbrandEtAl2005} and Stauffer \cite{Stauffer}, combined with the characterization of $\LS_+$-imperfect webs from \cite{EN2014} (Theorem \ref{webs}), we are able to show:

\begin{theorem}
\label{thm_N+imperfect_fcig}
All facet-defining $\LS_+$-perfect fuzzy circular interval graphs are cliques, odd holes or odd antiholes.
\end{theorem}

\begin{proof}
Let $G$ be a fuzzy circular interval graph such that it is a facet-defining $\LS_+$-perfect graph. 
If $G$ is a clique, the result is immediate. 
Otherwise, $G$ is the support graph of a clique family inequality with parameters $(\mathcal F, p)$ 
\[
(p-r)\sum_{i\in W} x_i+ (p-r-1) \sum_{i\in W_o} x_i \leq (p-r) \left\lfloor \frac{n}{p}\right\rfloor \!,
\]
associated with a web $W_n^{p-1}$ with $V(W_n^{p-1})\subset W$ (\cite{EisenbrandEtAl2005,Stauffer}).  

More precisely, if for any node $v\in V$, ${\cal{I}}_v=\{I\in {\cal{I}}: v\in I\}$, there exist $I_l(v)$ and $I_r(v)$ in ${\cal{I}}_v$ such that $I_l(v)\cup I_r(v)=\bigcup_{I \in {\cal{I}}_v} I$.
The clique family inequality associated with $W^{p-1}_n$ is the clique family inequality having parameters ${\cal{F}}=\{K(I_l(v)): v\in V(W_n^{p-1})\}$ and $p$ where $K(I_l(v))=\{u\in I_l(v): u \text{ is adjacent to } v\}$.

By Theorem \ref{webs}, $W_n^{p-1}$ is $\LS_+$-perfect if and only if it is an odd hole or an odd antihole.  

That is, since  $G$ is $\LS_+$-perfect then $n=2k+1$ and $p=2$ or $p=k\geq 3$. 
In both cases, $r=1$ follows.

Consider first the case in which $p=2$. Then the clique family inequality $(\mathcal F, p)$ takes the form
\[
\sum_{i\in W} x_i\leq \left\lfloor \frac{n}{p}\right\rfloor. 
\]


Suppose there exists $v\in W\setminus V(W_{2k+1}^1)$. Then, $v$ belongs to $s\geq 2$ consecutive cliques in $\cal{F}$ implying that $v$ is connected to exactly $s+1$ consecutive nodes in $W_{2k+1}^1$. Observe that $s\leq 3$ since $G$ is a claw-free graph. Then, if $s=2$ (resp. $s=3$) $G$ contains an odd subdivision of $G_{LT}$ (resp. $G_{EMN}$). Since $G$ is $LS_+$-perfect then $W=V(W_{2k+1}^1)$ or, equivalently,  $G=W_{2k+1}^1$.

Now suppose that $p=k\geq 3$. Let us call $\{1,2,\ldots,2k+1\}$ the nodes in $V(W_{2k+1}^{k-1})$.

Suppose there exists $v\in (W_o\cup W)\setminus V(W_{2k+1}^{k-1})$. Then $v$ belongs to at least $s\geq k-1$ consecutive cliques of the family $\mathcal F$. W.l.o.g we may assume that the $k-1$ of the $s$ consecutive cliques are the ones that contain the sets of nodes $\{1,\ldots,k\}$, $\{2,\ldots,k+1\}$, ... $\{k-1,\ldots,2k-2\}$. Then $v$ is connected to at least $2k-2$ consecutive nodes in $W_{2k+1}^{k-1}$. Moreover, since $G$ is quasi-line, $v$ is connected with at most $2k$ nodes. It follows that the subgraph of $G$ induced by $V(W_{2k+1}^{k-1})\cup \{v\}$ has stability number two, and from Theorem \ref{thm_N+imperfect_alpha2} it is $\LS_+$-imperfect. But from our assumption that $G$ is $\LS_+$-perfect, we conclude $W_o\cup W=V(W_{2k+1}^{k-1})$ or, equivalently, $G=W_{2k+1}^{k-1}$.
$\Box$
\end{proof}

As a consequence, every $\LS_+$-perfect fuzzy circular interval graph is a-perfect.
This verifies the $\LS_+$-Perfect Graph Conjecture for fuzzy circular interval graphs.
%

Since the class of quasi-line graphs 
divides into semi-line graphs and fuzzy circular interval graphs, we obtain as direct consequence: 

\begin{corollary}\label{Cor_quasi-line}
The $\LS_+$-Perfect Graph Conjecture is true for quasi-line graphs.
\end{corollary}

\subsection{Claw-free graphs that are not quasi-line} 

It is left to treat the case of claw-free graphs that are not quasi-line. Here, we distinguish two cases according to their stability number.

To treat the case of claw-free not quasi-line graphs $G$ with $\alpha(G) \geq 4$, we 
rely on the decomposition of such graphs into strips, where at most one strip is quasi-line and all the remaining ones are 5-wheel strips \cite{OPS2008}. By noting that 5-wheel strips of type 3 contain $G_{LT}$ and exhibiting $\LS_+$-imperfect line graphs in the other two cases, we are able to show:

\begin{theorem}
\label{thm_not_quasi-line_4}
Every facet-defining claw-free not quasi-line graph $G$ with $\alpha(G) \geq 4$ is $\LS_+$-imperfect.
\end{theorem}

\begin{proof}
Let $G$ be a facet-defining claw-free not quasi-line graph with $\alpha(G) \geq 4$. 
According to~\cite{OPS2008}, $G$ has a decomposition into strips, where at most one strip is quasi-line and all the remaining ones have stability number at most 3 and contain a 5-wheel each. 
Recall that there are only three types of 5-wheel strips, Fig.~\ref{Fig:5wheelStrips} shows the ``basic'' types, which can be extended by adding nodes belonging to the neighborhood of the 5-wheels \cite{OPS2008}. 

Since $G$ is not quasi-line, it contains at least one 5-wheel strip $G'$. 
If $G'$ is of type 3, then $G'$ contains $G_{LT}$, induced by the squared nodes indicated in Fig.~\ref{Fig:5wheelStrips}, and we are done. Hence, let $G'$ be of type 1 or 2.  

Note further that $G'$ is a proper subgraph of $G$ (by $\alpha(G') \leq 3$ but $\alpha(G) \geq 4$) and connected to $G-G'$ (since $G$ is facet-defining and, thus, cannot have a clique cutset by Chv\'atal \cite{Chvatal1975}).

According to the strip composition, there are nodes in $G-G'$ playing the role of the two simplicial nodes of $G'$ (the two black nodes in Fig.~\ref{Fig:5wheelStrips}), and they are connected by a path $P$ with nodes exclusively in $G-G'$ (again, since $G$ cannot contain a clique cutset). 

If $G'$ is of type 1, then 
$G$ has, as induced subgraph, a node stretching of $G_{EMN}$ (resp. of $G_{LT}$) if $P$ is even (resp. odd), see the squared nodes in Fig. \ref{Fig:5wheelStrip1}. 
 
\begin{figure}
\begin{center}
\includegraphics[scale=0.8]{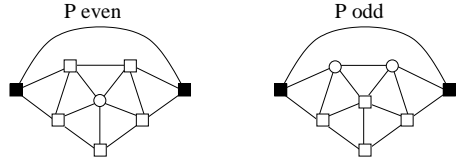}
\caption{$N_+$-imperfect subgraphs if $G'$ is of type 1.}
\label{Fig:5wheelStrip1}
\end{center}
\end{figure}

If $G'$ is of type 2, then 
$G$ has, as induced subgraph, a node stretching of $G_{LT}$ (resp. of $G_{EMN}$) if $P$ is even (resp. odd), see the squared nodes in Fig. \ref{Fig:5wheelStrip2}. 

\begin{figure}
\begin{center}
\includegraphics[scale=0.8]{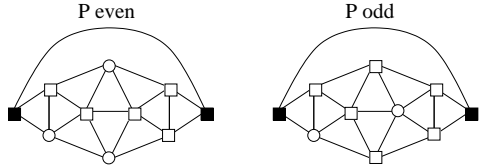}
\caption{$\LS_+$-imperfect subgraphs if $G'$ is of type 2.}
\label{Fig:5wheelStrip2}
\end{center}
\end{figure}

Hence, in all cases, $G$ contains an $\LS_+$-imperfect line graph and is itself $\LS_+$-imperfect. 
$\Box$
\end{proof}

For graphs having stability number three, there is no decomposition known yet. 
Relying only on the behavior of the stable set polytope under clique identification \cite{Chvatal1975} and the result on $\LS_+$-(im)perfect graphs from Theorem \ref{mnpmascero}, we can prove:

\begin{theorem}
\label{thm_not_quasi-line_3}
Every facet-defining claw-free not quasi-line graph $G$ with $\alpha(G) = 3$ is $\LS_+$-imperfect.
\end{theorem}

\begin{proof}
Let $G$ be a facet-defining claw-free graph with $\alpha(G) = 3$ that is not quasi-line. 
Then, there is a node $v$ in $G$ such that $G' = G[N(v)]$ cannot be partitioned into 2 cliques. 
Hence, in the complement $\overline G$ of $G$, the subgraph $\overline G'$ cannot be partitioned into two stable sets. 
Thus, $\overline G'$  is non-bipartite and contains an odd cycle. 
Let $C$ be the shortest odd cycle in $\overline G'$. 
Then $C$ is not a triangle (otherwise, $\overline C$ and $v$ induce a claw in $G$). 
Hence, $C$ is an odd hole (because it is an odd cycle of length $\geq 5$, but has no chords according to its choice). 

Therefore, $\overline C$ is an odd antihole in $G$; let $u_1, \ldots, u_{2k+1}$ denote it nodes and $u_iu_{i+k}$ be its non-edges. 
Furthermore, let $G_v = G[N(v) \cup \{v\}]$ and $W = N(v)-\overline C$.

From now on, we will use $\overline C$ to denote both, the node set and the odd antihole when it is clear from the context.
 
\begin{myclaim}\label{claim1}
$G_v$ has stability number 2 and \\[-6mm]
\begin{itemize}
\item either contains an $LS_+$-imperfect subgraph
\item or is the complete join of $v$, $\overline C$ and $W$.
\end{itemize}
\end{myclaim}
\noindent
Firstly, note that $N(v)$ does not contain a stable set of size 3 (otherwise, $G$ clearly contains a claw). 
Hence,  $\alpha(G_v) = 2$ follows. 
If $W = \emptyset$, we are done. 
If there is a node $w \in W$, then for each such node, either $w$ is completely joined to $\overline C$ or else the subgraph of $G$ induced by $w$ and $\overline C$ is $LS_+$-imperfect due to Theorem \ref{mnpmascero}. 
$\Diamond$\\

We are done if $G_v$ is $LS_+$-imperfect. 
Hence, assume in the sequel of this proof that $G_v$ is the complete join of $v$, $\overline C$ and $W$. 
Since $\alpha(G) = 3$ holds, $G_v$ is a proper subgraph of $G$. 
We partition the nodes in $G-G_v$ into 3 different subsets:\\[-6mm]
\begin{itemize}
\item $X$ containing all nodes from $G-G_v$ having a neighbor in $W$,
\item $Y$ having all nodes from $G-G_v$ having no neighbor in $W$, but a neighbor in~$\overline C$,
\item $Z$ containing all nodes from $G-G_v$ having no neighbor in $W \cup \overline C$.
\end{itemize}

\begin{myclaim}\label{claim2}
Every node $x \in X$ \\[-6mm]
\begin{itemize}
\item either induces together with $\overline C$ an $LS_+$-imperfect subgraph of $G$,
\item or is completely joined to $\overline C$.
\end{itemize}
\end{myclaim}
\noindent
No node $x \in X$ can belong to a stable set $S = \{x,u_i,u_{i+k}\}$ (otherwise, any neighbor $w \in W$ of $x$ induces together with $S$ a claw in $G$). 
Hence, for every $x \in X$, the subgraph $G[\overline C \cup \{x\}]$ has stability number 2 and is either $LS_+$-imperfect or an odd antiwheel by Theorem \ref{mnpmascero}. 
$\Diamond$\\

We are done if some node $x \in X$ yields an $LS_+$-imperfect graph $G[\overline C \cup \{x\}]$. 
Hence, assume in the sequel of this proof that $\overline C$ and $X$ are completely joined. 

\begin{myclaim}\label{claim3}
$X$ is a clique.
\end{myclaim}
\noindent
Otherwise, $G$ contains a claw containing some node $u_i \in \overline C$ as central node, $v$ and two non-adjacent nodes $x,x' \in X$. 
$\Diamond$\\

Let $G_X$ denote the subgraph of $G$ induced by $v$, $N(v)$ and $X$.

\begin{myclaim}\label{claim4}
We have $\alpha(G_X) = 2$.
\end{myclaim}
\noindent
We know already that $\alpha(G_v) = 2$ by Claim 1. 
If $G_X$ contains a stable set $S$ of size 3, then $x \in S$ for some $x \in X$. 
This implies $S \cap \overline C = \emptyset$ (recall that we assume that $X$ and $\overline C$ are completely joined). 
In addition, $v \not\in S$ (since $v$ is adjacent to all nodes in $W$ (so we would have $S \cap W = \emptyset$ if $v \in S$), but $S$ cannot contain 2 nodes from $X$ (since $X$ is a clique by Claim 3)). 
Finally, $S$ cannot contain 2 non-adjacent nodes $w,w' \in W$ (otherwise, any node $u_i \in \overline C$ induces with $S$ a claw in $G$). 
Hence, there is no such stable set $S$ in $G_X$. 
$\Diamond$\\

By $\alpha(G) = 3$ and $\alpha(G_X) = 2$, there is a node in $Y \cup Z$. 
We conclude that $Y \neq \emptyset$ (otherwise, $X$ would constitute a clique cutset of $G$, separating $Z$ from $G'$, a contradiction to $G$ facet-defining by Chv\'atal \cite{Chvatal1975}.

\begin{myclaim}\label{claim5}
$W$ induces a clique.
\end{myclaim}
\noindent
Otherwise, $G$ contains a claw induced by some node $y \in Y$, a neighbor $u_i \in \overline C$ of $y$, and two non-adjacent nodes $w,w' \in W$. 
$\Diamond$\\

Hence, 
$G_v$ is in fact the complete join of a clique $Q = \{v\} \cup W$ and $\overline C$. 
In addition, $X$ is a clique and completely joined to $\overline C$, $Y$ is non-empty, and there is no edge between $Q$ and $Y$. We further obtain:

\begin{myclaim}\label{claim6}
Every node $y \in Y$ is completely joined to $X$.
\end{myclaim}
\noindent
Otherwise, $G$ contains a claw induced by some node $y \in Y$, a neighbor $u_i \in \overline C$ of $y$, $v$ and a non-neighbor $x \in X$ of $y$. 
$\Diamond$\\

Note that, according to Theorem \ref{mnpmascero}, each node $y \in Y$ has three possibilities for its connections to $\overline C$:\\[-6mm]
\begin{itemize}
\item either $y$ induces together with $\overline C$ an $LS_+$-imperfect subgraph of $G$, 
\item or $y$ is completely joined to $\overline C$, 
\item or $y$ belongs to a stable set $S_y = \{y,u_i,u_{i+k}\}$ 
\end{itemize}
(recall that, by Theorem \ref{mnpmascero}, whenever $\{y\} \cup \overline C$ has stability number 2, it is either $LS_+$-imperfect or an odd antiwheel).
If a node $y \in Y$ gives rise to an $LS_+$-imperfect subgraph of $G$, we are done. 
Hence, assume in the sequel of this proof that $Y$ is partitioned into two subsets $Y_*$ and $Y_S$ containing all nodes $y$ that are completely joined to $\overline C$ resp. belong to a stable set $S_y = \{y,u_i,u_{i+k}\}$.
We next show: 

\begin{myclaim}\label{claim7}
$Y_S \neq \emptyset$.
\end{myclaim}
\noindent
Assume to the contrary that we have $Y = Y_*$. 
Then, $Y$ also induces a clique (otherwise, there is a claw in $G$ induced by $v$, some node $u_i \in \overline C$ and two non-adjacent nodes $y,y' \in Y$). 
This implies that $G[G_v \cup X \cup Y]$ has stability number 2 (by $\alpha(G_X) = 2$ due to Claim 4, $Y$ completely joined to $X$ due to Claim 6, and $Y = Y_*$ clique).
Thus, $Z$ is non-empty (because $\alpha(G) = 3$). 
Hence, $G$ contains a clique cutset $X \cup Y$, separating $Z$ from $G_v$ (recall that every node in $Z$ has only neighbors in $X$ or $Y$, but not in $G_v$), a contradiction to $G$ facet-defining by Chv\'atal \cite{Chvatal1975}. 
Therefore, we conclude that $Y = Y_*$ cannot hold. 
$\Diamond$\\

Having ensured the existence of a stable set $S_y = \{y,u_i,u_{i+k}\}$ in $G$, we next observe:

\begin{myclaim}\label{claim8}
$X = \emptyset$.
\end{myclaim}
\noindent
Otherwise, $G$ contains a claw induced by $S_y$ and any node $x \in X$ (recall: we assume that $X$ and $\overline C$ are completely joined (otherwise, $G$ is $LS_+$-imperfect by Claim 2), and have that $X$ and $Y$ are completely joined by Claim 6). 
$\Diamond$\\

This implies particularly that no node outside $G_v$ has a neighbor in $W$. 
We next study the connections between $\overline C$ and $Y$ in more detail and obtain the following important fact:

\begin{myclaim}\label{claim9}
$\overline C = C_5$ and each node $y \in Y_S$ has exactly two consecutive neighbors on $\overline C$.
\end{myclaim}
\noindent
Consider some node $y \in Y_S$ and the stable set $S_y = \{y,u_i,u_{i+k}\}$. 
By construction of $Y$, $y$ has a neighbor $u_j \in \overline C$. 
This node $u_j$ (and any further neighbor of $y$ in $\overline C$) cannot be a common neighbor of $u_i$ and $u_{i+k}$ (otherwise, $u_j$ induces together with $S_y$ a claw in $G$). Hence, $u_j$ equals either $u_{i-1}$ (which is not adjacent to $u_{i+k}$) or else $u_{i+k+1}$ (which is not adjacent to $u_{i}$). 
W.l.o.g., say that $y$ has $u_{i-1}$ as neighbor in $\overline C$. 
Then $\overline C = C_5$ follows because for any $k \geq 3$, the graph induced by $y$ and $\overline C$ contains a claw with center $u_{i-1}$ and the nodes $u_{i+1},u_{i+k+2},y$ (or else $u_{i+1}$ or $u_{i+k+2}$ induce with $S_y$ a claw if $y$ is adjacent to $u_{i+1}$ or $u_{i+k+2}$). 

Moreover, we observe that $y$ is also adjacent to $u_{i-2}$ (otherwise, there is a claw with center $u_{i-1}$ and the nodes $u_{i},u_{i-2},y$). 
This shows the assertion that each node $y \in Y_S$ has exactly two consecutive neighbors on $\overline C = C_5$.
$\Diamond$\\

We next observe:

\begin{myclaim}\label{claim10}
$Z$ induces a clique.
\end{myclaim}
\noindent
Otherwise, $G$ contains a stable set of size 4, consisting of two non-adjacent nodes in $\overline C$ and two non-adjacent nodes in $Z$ (recall: by definition of $Z$, there is no edge between $\overline C$ and $Z$). 
$\Diamond$\\

Hence, so far we have the following: 
$G_v$ is the complete join of a clique $Q = \{v\} \cup W$ and $\overline C = C_5$. 
$G-G_v$ is partitioned into two subsets $Y$ and $Z$ 
where\\[-6mm]
\begin{itemize}
\item $Y$ is non-empty and partitions into two subsets $Y_*$ and $Y_S$ consisting of all nodes $y$ that are completely joined to $\overline C = C_5$ resp. belong to a stable set $S_y = \{y,u_i,u_{i+k}\}$ and have exactly two consecutive neighbors on the $C_5$;
\item $Z$ induces a clique and no node in $Z$ has a neighbor in $G_v$. 
\end{itemize}
We continue to explore the composition of $Y$ and its connections to $\overline C = C_5$:

\begin{myclaim}\label{claim11}
If two nodes $y,y' \in Y$ share a same neighbor $u_j \in \overline C$, then $y$ and $y'$ are adjacent.
\end{myclaim}
\noindent
Otherwise, $G$ contains a claw with center $u_j$ and the nodes $v,y,y'$. 
$\Diamond$

\begin{myclaim}\label{claim12}
There are at least two nodes in $Y_S$.
\end{myclaim}
\noindent
By Claim 7, there is a node $y \in Y_S$. 
Then $Y \neq \{y\}$ follows (otherwise, the only two and consecutive neighbors of $y$ on $\overline C = C_5$ (by Claim 9) form a clique cutset in $G$, separating $y$ from $Q$, a contradiction to $G$ facet-defining by Chv\'atal \cite{Chvatal1975}). 
If all nodes from $Y - \{y\}$ belong to $Y_*$, then $Y$ induces a clique by Claim 11 (because all share a common neighbor in $\overline C$), and $Y - \{y\}$ together with the two and consecutive neighbors of $y$ on $\overline C$ form a clique cutset in $G$, separating $y$ from $Q$, again a contradiction to $G$ facet-defining. 
Hence, $Y_S$ contains at least two nodes.
$\Diamond$\\

Using similar arguments, we next show:

\begin{myclaim}\label{claim13}
Not all nodes in $Y_S$ have the same two consecutive neighbors on $\overline C$. 
\end{myclaim}
\noindent
Otherwise, $G$ has a clique cutset (consisting of $Y_*$ and the two common, consecutive neighbors on $\overline C$ of all nodes in $Y_S$), separating $Y_S$ from $Q$, again a contradiction to $G$ facet-defining. 
$\Diamond$

\begin{myclaim}\label{claim14}
If all nodes in $Y_S$ share a common neighbor $u_i$ on $\overline C$, then $G$ contains $G_{EMN}$ as induced subgraph. 
\end{myclaim}
\noindent
By assumption, we have only three types of nodes in $Y$: 
nodes $y$ with $N_{\overline C}(y) = \{u_{i-1},u_i\}$, 
nodes $y'$ with $N_{\overline C}(y') = \{u_i,u_{i+1}\}$, and 
nodes $y_* \in Y_*$. 
$Y$ induces a clique (by Claim 11) and $Z = \emptyset$ follows (otherwise, $Y$ is a clique cutset separating $Z$ from $Q$, again a contradiction to $G$ facet-defining). 
Claim 12 combined with Claim 13 shows that there is at least one node $y$ adjacent to $u_{i-1},u_i$ and at least one node $y'$ adjacent to $u_i,u_{i+1}$. 
Moreover, there is also at least one node $y_* \in Y_*$ (otherwise, $G$ equals the gear (induced by $v$, $\overline C$, $y$ and $y'$, see Figure \ref{Fig_N+P_alpha3_claim14}) with possible replications of $v$, $y$, $y'$ and is not facet-defining, a contradiction). 
Hence, $G$ contains $G_{EMN}$ (induced by $y_*$ and the nodes $v$, $u_{i-1}$, $u_{i+1}$, $y$, and $y'$, see Figure \ref{Fig_N+P_alpha3_claim14}). 
$\Diamond$\\

\begin{figure}
\begin{center}
\includegraphics[scale=0.8]{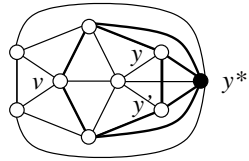}
\caption{Subgraph induced by $\overline C$ and $v$, $y$, $y'$, $y_*$ (removing $y_*$ yields the gear, the bold edges indicate the $G_{EMN}$).}
\label{Fig_N+P_alpha3_claim14}
\end{center}
\end{figure}

This shows that $G$ is either $LS_+$-imperfect (and we are done) or $Y_S$ has two nodes with distinct neighbors on $\overline C$. Let us assume the latter. 

\begin{myclaim}\label{claim15}
If two nodes $y y' \in Y_S$ with distinct neighbors on $\overline C$ are adjacent, then $G$ contains $G_{LT}$ as induced subgraph. 
\end{myclaim}
\noindent
W.l.o.g., let $N_{\overline C}(y) = \{u_{1},u_2\}$ and $N_{\overline C}(y') = \{u_3,u_{4}\}$. 
If $y$ is adjacent to $y'$, then $u_1$, $y$, $y'$, $u_4$, $u_5$ induce together with $v$ a $G_{LT}$. 
$\Diamond$\\

Hence, assume that no two nodes in $y y' \in Y_S$ with distinct neighbors on $\overline C$ are adjacent (otherwise, we are done).

\begin{myclaim}\label{claim16}
If $Z \neq \emptyset$, then $G$ contains 
a node stretching of $G_{EMN}$ as induced subgraph.
\end{myclaim}
\noindent
If there is a node $z \in Z$, then $z$ is adjacent to every node in $Y_S$ (otherwise, there is a node $y \in Y_S$ such that $z$ together with $S_y$ forms a stable set of size 4). 
Recall that we assume that $Y_S$ contains two non-adjacent nodes $y$, $y'$ with distinct neighbors on $\overline C$. 
Then $z$ together with $y$, $y'$ and $\overline C$ induce a node stretching of $G_{EMN}$.
$\Diamond$\\

Hence, we are done if $Z \neq \emptyset$ (because $G$ contains an $LS_+$-imperfect subgraph). 
So let us assume $Z = \emptyset$ from now on. 

Furthermore, assume w.l.o.g. that $y$ with $N_{\overline C}(y) = \{u_{1},u_2\}$ and $y'$ with $N_{\overline C}(y') = \{u_3,u_{4}\}$ is a pair of nodes in $Y_S$ having distinct neighbors on $\overline C$. 
Since $G$ is facet-defining (and, thus, without clique cutset), there is a path connecting $y$ and $y'$; let $P$ denote the shortest such path.

Note that $P$ has length $\geq 2$ (recall: $y$ and $y'$ are supposed to be non-adjacent, otherwise $G$ contains a $G_{LT}$ by Claim 15). 

\begin{myclaim}\label{claim17}
If $P$ has length $2$, then $G$ is $LS_+$-imperfect.
\end{myclaim}
\noindent
So, let $P$ have length $2$ and denote by $t$ its only internal node. 
Then $t \in Y$ follows (by $Z = \emptyset$ and because there is no common neighbor of $y$ and $y'$ in $\overline C$). 
We conclude that $t \not\in Y_*$ holds (otherwise, $G$ has a claw induced by $t$ and $S_y$) and that $t$ and $u_5$ are non-adjacent (otherwise, $G$ has a claw induced by $t$ and $y$, $y'$, $u_5$). 
In fact, $N_{\overline C}(t) = \{u_{2},u_3\}$ follows by our assumption that no two nodes in $Y_S$ with distinct neighbors on $\overline C$ are adjacent (otherwise, $G$ contains a $G_{LT}$ by Claim 15). 

Since the graph induced by $G_v$ together with $y,t,y'$, called a 3-gear (see Figure \ref{Fig_alpha3_claim17}), is not facet-defining, there must be another node $y''$ in $Y$ (recall: we have $Z = \emptyset$). 
We are done if there is a node $y'' \in Y_*$ (because $G$ contains a $G_{EMN}$ induced by $u_1, y, t, u_3, v$ and $y''$ in this case). 
Hence, assume $y'' \in Y_S$. 

W.l.o.g., let $N_{\overline C}(y'') = \{u_{4},u_5\}$ (note: the case $N_{\overline C}(y'') = \{u_{1},u_5\}$ is symmetric, and in all other cases, $G$ is still a 3-gear with some replicated nodes, thus not facet-defining). 
Then $y''$ and $y'$ are adjacent by Claim 11. 
If $y''$ is also adjacent to $y$ or $t$, then we are done since then $G$ contains a $G_{LT}$ by Claim 15. 
If $y''$ is neither adjacent to $y$ nor to $t$, then $G$ contains an $LS_+$-imperfect line graph induced by $u_1, y, t, u_3, v$ and $y',y'',u_3$ (being a node stretching of $G_{EMN}$). 
Note that $G$ still contains one of the above $LS_+$-imperfect subgraphs if $G$ contains more nodes than considered so far. 
$\Diamond$\\

\begin{figure}
\begin{center}
\includegraphics[scale=0.6]{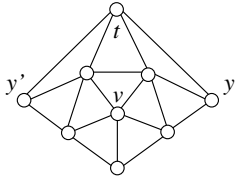}
\caption{Subgraph induced by $\overline C$ and $v$, $y$, $y'$, $t$, called a 3-gear.}
\label{Fig_alpha3_claim17}
\end{center}
\end{figure}

Thus, we are done if $P$ has length $2$. 
Let us finally assume that $P$ has length $\geq 3$ and $y$ and $y'$ have no common neighbor (recall: $P$ is a shortest path connecting them). 

Then $Y_* = \emptyset$ follows (because each node $y'' \in Y_*$ shares a common neighbor with $y$ and $y'$ on $\overline C$ and is, thus, adjacent to both $y$ and $y'$ by Claim 11, a contradiction to the choice of $P$ as shortest path connecting them). 

Moreover, there is a neighbor $\overline y$ of $y$ in $Y_S$ (otherwise, $\{u_{1},u_2\}$ forms a clique cutset separating $y$ from $Q$, a contradiction to $G$ facet-defining by Chv\'atal \cite{Chvatal1975}). 
In addition, $N_{\overline C}(\overline y)$ is different from $\{u_{3},u_4\}$ and from $\{u_{4},u_5\}$ (otherwise, $y$ and $y'$ would be a pair of nodes with distinct neighbors on $\overline C$ and connected by a path of length 2, hence $G$ is $LS_+$-imperfect by Claim 17 and we are done). 

We clearly have $N_{\overline C}(\overline y) \neq \{u_{2},u_3\}$ (otherwise, $\overline y$ adjacent to $y'$ follows by Claim 11 and we have again a pair of nodes with distinct neighbors on $\overline C$ and connected by a path of length 2, hence $G$ is $LS_+$-imperfect by Claim 17 and we are done). 

If there is no node in $Y_S$ having $\{u_{1},u_5\}$ as neighbors on $\overline C$, then $\{u_{1},u_2\}$ forms still a clique cutset separating $y$ from $Q$, again a contradiction. 
Hence, let $N_{\overline C}(\overline y) = \{u_{1},u_5\}$. 

Then, there is no node in $Y_S$ having $\{u_{4},u_5\}$ as neighbors on $\overline C$ (this node $y''$ would be adjacent to $\overline y$ by Claim 11, so that $\overline y$ would be a common neighbor of $y$ and $y''$, leading to an $LS_+$-imperfect subgraph of $G$ by Claim 17  and we are done). 
Similarly, there is no node in $Y_S$ having $\{u_{3},u_4\}$ as neighbors on $\overline C$.

This finally implies that $\{u_{3},u_4\}$ is a clique cutset separating $y'$ from $Q$, a contradiction to $G$ facet-defining by Chv\'atal \cite{Chvatal1975}). 
That all further nodes of $G$ are either replicates of $y$, $y'$ or $\overline y$ (and $\{u_{3},u_4\}$ remains a clique cutset in all cases) finishes the proof. 
$\Box$
\end{proof}


Hence, the only facet-defining subgraphs $G'$ of $\LS_+$-perfect claw-free not quasi-line graphs $G$ with $\alpha(G) = 3$ have $\alpha(G') = 2$ and are, by Theorem \ref{thm_N+imperfect_alpha2}, cliques, odd antiholes or their complete joins. 
We conclude that $\LS_+$-perfect facet-defining claw-free not quasi-line graphs $G$ with $\alpha(G) = 3$ are joined $a$-perfect and, thus, the $\LS_+$-Perfect Graph Conjecture is true for this class.
%

This together with Theorem \ref{thm_not_quasi-line_3} shows that the only facet-defining subgraphs $G'$ of $\LS_+$-perfect claw-free not quasi-line graphs $G$ with $\alpha(G) \geq 4$ have $\alpha(G') = 2$ and are, by Theorem \ref{thm_N+imperfect_alpha2}, cliques, odd antiholes or their complete joins. Thus, every $\LS_+$-perfect claw-free not quasi-line graph $G$ with $\alpha(G) \geq 4$ is joined a-perfect and, thus, the $\LS_+$-Perfect Graph Conjecture holds true for this class.
%

Combining Corollary \ref{Cor_alpha2_2} with the above results shows that all $\LS_+$-perfect claw-free but not quasi-line graphs are joined $a$-perfect and we obtain:

\begin{corollary}\label{Cor_not_quasi-line}
The $\LS_+$-Perfect Graph Conjecture is true for all claw-free graphs that are not quasi-line.
\end{corollary}

Finally, we obtain our main result (Theorem \ref{thm_main}) as direct consequence of Corollary \ref{Cor_quasi-line} and Corollary \ref{Cor_not_quasi-line}:
The $\LS_+$-Perfect Graph Conjecture is true for all claw-free graphs.

\section{Conclusion and future research}

The context of this work was the study of 
$\LS_+$-perfect graphs, i.e., graphs where a single application of the Lov\'asz-Schrijver PSD-operator $\LS_+$ to the edge relaxation yields the stable set polytope. 
Hereby, we are particularly interested in finding an appropriate polyhedral relaxation $P(G)$ of $\stab(G)$ that coincides with $\LS_+(G)$ and $\stab(G)$ if and only if $G$ is $\LS_+$-perfect. 
An according conjecture has been recently formulated ($\LS_+$-Perfect Graph Conjecture); here we verified it for the well-studied class of claw-free graphs (Theorem \ref{thm_main}).

For that, it surprisingly turned out that it was not necessary to make use of the description of STAB$(G)$ for claw-free not quasi-line graphs $G$ 
\begin{itemize}
\item with $\alpha(G) = 2$ (by Cook, see~\cite{Shepherd1994}),
\item with $\alpha(G) = 3$ (by P\^echer, Wagler \cite{PecherWagler2010}),
\item with $\alpha(G) \geq 4$ (by Galluccio, Gentile, Ventura \cite{GGV2008,GGV2014a,GGV2014b}).
\end{itemize}

From the presented results and proofs, we can draw some further conclusions. 
First of all, we can determine the subclass of joined $a$-perfect graphs to which all $\LS_+$-perfect claw-free graphs belong to. 
In \cite{KosterWagler_IRII}, it is suggested to call a graph $G$ \emph{$m$-perfect} if the only facets of $\stab(G)$ are associated with cliques and minimally imperfect graphs. According to \cite{CPW_2009}, $G$ is \emph{joined $m$-perfect} if $\stab(G)$ is given only by facets associated with cliques, minimally imperfect graphs and their complete joins. 
Theorem \ref{hipomatch} together with the results from Section \ref{results} provide the complete list of all facet-defining $\LS_+$-perfect claw-free graphs: 
\begin{itemize}
\item cliques,
\item odd holes and odd antiholes,
\item complete joins of odd antihole(s) and a (possibly empty) clique.
\end{itemize}
Hence, we conclude:
\begin{corollary}
All $\LS_+$-perfect claw-free graphs are joined $m$-perfect.
\end{corollary}

Among these possible facets, only complete joins of odd antihole(s) and a non-empty clique are non-rank. This directly implies:
\begin{corollary}
A rank-perfect $\LS_+$-perfect claw-free graph has as only facet-defining subgraphs cliques, odd holes, odd antiholes, or complete joins of the latter.
\end{corollary}
Note that Galluccio and Sassano provided in \cite{GalluccioSassano97} a complete characterization of the rank facet-defining claw-free graphs: they either belong to one of the following three families of rank-minimal graphs
\begin{itemize}
\item cliques, 
\item partitionable webs $W^{\omega - 1}_{\alpha \omega + 1}$ (where $\alpha$ and $\omega$ stand for stability and clique number, resp.), 
\item line graphs of minimal 2-connected hypomatchable graphs $H$ (where $H-e$ is not hypomatchable anymore for any edge $e$), 
\end{itemize}
or can be obtained from them by means of two operations, sequential lifting and complete join.
Our results show: an $\LS_+$-perfect claw-free graph $G$ has, besides cliques, only odd holes and odd antiholes as rank-minimal subgraphs; cliques are the only subgraphs in $G$ that can be sequentially lifted to larger rank facet-defining subgraphs, where complete joins can only be taken of odd antiholes.

Note further that, besides verifying the $\LS_+$-Perfect Graph Conjecture for claw-free graphs, we obtained 
the complete list of all minimally $\LS_+$-imperfect claw-free graphs. 
In fact, the results in \cite{BENT2011,EN2014,ENW2014} imply that the following graphs are minimally $\LS_+$-imperfect:
\begin{itemize}
\item 
graphs $G$ with $\alpha(G)=2$ such that $G-v$ is an odd antihole for some node $v$, not completely joined to $G-v$,
\item the web $W_{10}^2$,
\item $\LS_+$-imperfect line graphs (which are all node stretchings of $G_{LT}$ or $G_{EMN}$). 
\end{itemize}
Our results from Section \ref{results} on facet-defining $\LS_+$-perfect claw-free graphs imply that they are the only minimally $\LS_+$-imperfect claw-free graphs.

Finally, the subject of the present work has parallels to the well-developed research area of perfect graph theory also in terms of polynomial time computability. In fact, it has the potential of reaching even stronger results due the following reasons. Recall that calculating the value $$\eta_+(G) = \max \uno^Tx, x \in \LS_+(G)$$ can be obtained with arbitrary precision in polynomial  time for every graph $G$, even in the weighted case, by \cite{LovaszSchrijver1991}. Thus, the stable set problem can be solved in polynomial time for a strict superset of perfect graphs, the $\LS_+$-perfect graphs, by $\alpha(G) = \eta_+(G)$. 
Hence, our future lines of research include to find 
\begin{itemize}
\item new families of graphs where the conjecture holds (e.g., by characterizing the minimally $\LS_+$-imperfect graphs within the class), 
\item new subclasses of $\LS_+$-perfect or joined a-perfect graphs, 
\item classes of graphs $G$ where $\stab(G)$ and $\LS_+(G)$ are ``close enough'' to have $\alpha(G) = \lfloor \eta_+(G) \rfloor$.
\end{itemize}
In particular, the class of graphs $G$ with $\alpha(G) = \lfloor \eta_+(G) \rfloor$ can be expected to be large since $\LS_+(G)$ is a much stronger relaxation of $\stab(G)$ than $\thbo(G)$.
In all cases, the stable set problem could be approximated with arbitrary precision in polynomial time in these graph classes by optimizing over $\LS_+(G)$. Finally, note that $\LS_+(P(G))$ with
\[
\stab(G) \subseteq P(G) \subseteq \estab(G)
\]
clearly gives an even stronger relaxation of $\stab(G)$ than $\LS_+(G)$. 
However, already approximating with arbitrary precision over $\LS_+(\qstab(G))$ cannot be done in polynomial time anymore for all graphs $G$ by \cite{LovaszSchrijver1991}. 
Hence,  $\LS_+$-perfect graphs or their generalizations satisfying $\alpha(G) = \lfloor \eta_+(G) \rfloor$ are the most promising cases in this context.

\end{document}